\documentclass[a4paper,12pt]{amsart}

\usepackage{amsfonts}
\usepackage{amsmath}
\usepackage{amssymb}
\usepackage{textcomp}
\usepackage{mathrsfs}
\usepackage{hyperref}
\numberwithin{equation}{section}

\renewcommand{\aa}{\overset{{}_{{}_\circ}}{\mathrm{a}}}

\renewcommand{\a}{\alpha}

\newcommand{\g}{\gamma}
\newcommand{\G}{\Gamma}
\renewcommand{\d}{\delta}
\newcommand{\D}{\Delta}

\newcommand{\z}{\zeta}

\renewcommand{\l}{\lambda}
\renewcommand{\L}{\Lambda}
\newcommand{\m}{\mu}
\newcommand{\n}{\nu}
\newcommand{\x}{\xi}

\newcommand{\s}{\sigma}

\newcommand{\vs}{\varsigma}
\renewcommand{\t}{\tau}
\newcommand{\f}{\phi}

\renewcommand{\O}{\Omega}

\newcommand{\C}{{\mathbb C}}
\newcommand{\R}{{\mathbb R}}

\newcommand{\T}{{\mathbb{T}}}

\newcommand{\bb}{{\mathbf b}}

\newcommand{\cb}{{\mathbf c}}

\newcommand{\kb}{{\mathbf k}}

\newcommand{\rb}{{\mathbf r}}

\newcommand{\Gb}{{\mathbf G}}

\newcommand{\Ib}{{\mathbf I}}

\newcommand{\Kb}{{\mathbf K}}

\newcommand{\Mb}{{\mathbf M}}

\newcommand{\Sbb}{{\mathbf S}}

\newcommand{\aF}{\mathfrak a}

\newcommand{\SF}{\mathfrak S}

\newcommand{\Ac}{{\mathcal A}}
\newcommand{\Bc}{{\mathcal B}}
\newcommand{\Cc}{{\mathcal C}}

\newcommand{\Kc}{{\mathcal K}}
\newcommand{\Lc}{{\mathcal L}}

\newcommand{\Rc}{{\mathcal R}}

\newcommand{\diag}{{\rm diag}\,}

\newcommand{\grad}{{\rm grad}\,}

\renewcommand{\div}{{\rm div}\,}

\newcommand{\pO}{\partial\Omega}
\newcommand{\Ran}{\hbox{{\rm Ran}}\,}

\newcommand{\pp}{\pmb{\partial}}

\newcommand{\bG}{\pmb{\G}}
\newtheorem{theorem}{Theorem}[section]

\newtheorem*{theorem*}{Theorem}

\theoremstyle{definition}
\newtheorem{definition}[theorem]{Definition}

\theoremstyle{remark}
\newtheorem{remark}[theorem]{Remark}

\date{}

\begin{document}

\title[Neumann-Poincar\'e operator]{Spectral properties of  the Neumann-Poincar\'e operator in 3D elasticity}
\author{Yoshihisa Miyanishi}
\address{Center for Mathematical Modeling and Data Science, Osaka University, Japan}
\email{miyanishi@sigmath.es.osaka-u.ac.jp}

\author{Grigori Rozenblum }

\address{Chalmers University of Technology and The University of Gothenburg, Sweden; St.Petersburg State University Dept. Math. Physics, St.Petersburg, Russia.}

\email{grigori@chalmers.se}


\begin{abstract}
We consider the double layer potential (Neumann-Poincar\'e) operator
appearing in 3-dimensional elasticity. We show that the recent result about the polynomial compactness of this operator for the case of a homogeneous media follows without additional calculations from previous considerations by Agranovich et.al., based upon pseudodifferential operators. Further on, we define the NP operator for the case of a nonhomogeneous isotropic media and show that its properties depend crucially  on the character of non-homogeneity. If the Lam\'e parameters are constant along the boundary, the NP operator is still polynomially compact. On the other hand, if these parameters are not constant, two or more intervals of continuous spectrum may appear, so the NP operator ceases to be polynomially compact. However, after a certain modification, it becomes polynomially compact again. Finally, we evaluate the rate of convergence of discrete eigenvalues of the NP operator to the tips of the essential spectrum.
\end{abstract}
\maketitle
\section{Introduction}
This article was initially inspired by the recent paper \cite{3D} in this journal, where the polynomial compactness property for the Neumann-Poincar\'e (NP) operator for the 3D elasticity problem was established and (all) three points of its essential spectrum were found. The above paper used rather lengthy calculations in order to represent the NP operator as a pseudodifferential one, up to a weaker term. It turns out that one might have arrived at the main results of \cite{3D} in a much more simple way, using certain explicit  formulas derived quite long ago by M. Agranovich et.al. in \cite{AgrLame}. Further on, the papers above dealt with the case of an elastic body composed  of a homogeneous isotropic material. However we can show now that the machinery of pseudodifferential operators and the corresponding layer potentials as applied to strongly elliptic systems, developed in \cite{McLean} and further on by M. Agranovich in \cite{Agrgen}, enables one to handle the case of a  nonhomogeneous isotropic  material as well. Of course, the notion of the NP operator should be generalized to this case in a proper way, since the fundamental solution for the Lam\'e system cannot be expressed explicitly any more. The proper generalization is also based upon the constructions in \cite{McLean}, \cite{Agrgen}.  When studying the spectral structure of the resulting NP operator, we find out that  the polynomial compactness property, with the same three points of the essential spectrum, is still valid, as long as the Lam\'e parameters are constant on the boundary. If these parameters are only locally constant on the boundary (thus they may be different on different connected components of the boundary), there may exist an odd number of points of the essential spectrum, no more than twice the number of components of the boundary plus one. The NP operator is still polynomially compact but the expression for the polynomial  depends now on the topology of the boundary.

 A completely different picture arises when the Lam\'e parameters are not locally constant on the boundary. In this case the NP operator is not polynomially compact any more. Its essential spectrum consists of the point zero, an even number of intervals of the continuous spectrum, and, probably, some more, an even number, of isolated points.
Additionally, there may exist some eigenvalues of finite multiplicity, finitely or infinitely many of them. We can at the moment say nothing  about eigenvalues embedded into the essential spectrum. As for the eigenvalues outside the essential spectrum,  necessarily  converging only to the tips of the latter, we evaluate the rate of this convergence.  Of more use for applications might turn out to be the notion of the \emph{modified} NP operator, introduced here, which \emph{is} polynomially compact, again with three points of the essential spectrum, even for variable Lam\'e parameters.

In order to avoid technical complications, we assume that the boundary is smooth and that the Lam\'e parameters  are infinitely smooth up to the boundary.

In Sect.2 we recall the Lam\'e equations for the homogeneous and non-homogeneous media and discuss the notion of the NP operator, as well as demonstrate that in the homogeneous case the results of \cite{3D} follow easily from the considerations in \cite{AgrLame}. In Sect 3 we consider the nonhomogeneous case, introduce the NP operator and discuss its  pseudodifferential representation; then we recall the properties of the essential spectrum of zero order PsDO and prove our main theorems. The modified NP operator is considered here as well. The next section is devoted to the study of eigenvalues of the NP operator, where we show how the rate of their convergence to the tips of the essential spectrum depends on the behavior of the Lam\'e parameters near critical points of the boundary.  A discussion of related results is placed in the last section.

The recent renewal  of interest in the spectral properties of NP operators in different settings has arisen in relation to the plasmon resonance, see, e.g., \cite{IOP}, \cite{Ammari}, \cite{LiLiu}  and voluminous literature cited there. We  discuss briefly the consequences of our spectral results for this latter topic.
In the paper we consider the 3-dimensional case in detail. As it concerns the essential spectrum, the two-dimensional case does not require any additional machinery and can be treated in a similar way. However the behavior of eigenvalues is rather different. We discuss   some recent results in this direction.
 \section{The Lam\'e equations and the Lam\'e operator}\label{SecLameConst}
 \subsection{Lam\'e equations}
 The elastostatic Lam\'e equations in a domain $\O\subset\R^d$, $d=2,3,$ with infinitely smooth boundary $\pO$ for an isotropic body $\O$, is the system of $d$ equations having the form
\begin{equation}\label{Lame}
    \Lc u\equiv \Lc_{\m,\l}=\div(\m \grad u)+\grad((\l+\m)\div u)=0, x=(x_1,\dots, x_d)\in\O, u=(u_1,\dots,u_d)^\top,
\end{equation}
where $\l,\m$ are called Lam\'e parameters (see, e.g., \cite{KuprPot}, \cite{Kupr79}, \cite{Maugin}). If the functions $\l(x),\m(x)$ are, in fact, some constants, the body is called homogeneous, and it is nonhomogeneous otherwise. The homogeneous case is, of course, considerably better studied, since the fundamental solution is known explicitly - it is called the Kelvin matrix $\G(x)=\G_{\l,\m}(x)$. We do not need its explicit expression, so we do not present it here; one can find it, e.g., in \cite{AgrLame} or \cite{IOP}.

The Lam\'e parameters $\l,\m$ are always supposed to satisfy the so called strong convexity conditions
\begin{equation}\label{Convexity}
    \mu>0, d\l+2\m>0.
\end{equation}

In the nonhomogeneous case, we suppose that the (now) functions $\l(x),\m(x)$ are infinitely differentiable in $\O$, up to the boundary, and that conditions \eqref{Convexity} are satisfied uniformly, i.e.,
\begin{equation}\label{ConvexityU}
    \mu(x)>\d>0, d\l(x)+2\m(x)>\d>0;
\end{equation}
so they can be extended to a neighborhood of  $\overline{\O}$ (for example, to the whole space) as infinitely differentiable functions with bounded derivatives, subject to the same kind of uniform estimate \eqref{ConvexityU}.

\subsection{The NP operator and polynomial compactness for the homogeneous media}\label{SubNP}
The Lam\'e equations may be appended by different kinds of boundary conditions.  The `conormal derivative' is involved in some of them. The conormal derivative at the boundary point $x\in\pO$,  is defined as
\begin{equation*}%\label{conormal}
    \pp=\pp_{\n_x}=\l(x)(\nabla\cdot u)\n_x +2\mu(x)(\widehat{\nabla}u)\cdot\n_x,
\end{equation*}
$\widehat{\nabla}u,$ being the `symmetric gradient', the matrix
\begin{equation}\label{symgrad}
    \widehat{\nabla}u=\frac12(\nabla u+(\nabla u)^\top),
\end{equation}
with  $\n_x=(\n^1_x,\dots,\n^d_x)$ denoting the outward unit normal to $\pO$ at $x$.

A more explicit description of the conormal, (stress, traction) operator $\pp$, given, e.g., in \cite{AgrLame}, \cite{Kupr79}, represents $\pp$ as
the operator $d\times d$ matrix with  components
\begin{equation*}%\label{stress components}
    T_{jk}\equiv T_{jk}(\partial_x,\n_x)=\l(x) \n^{j}_x \partial_k(x)+\l (x)\n^k_x\partial_j+(\l(x)+\m(x))\delta_{jk}\partial_{\n_x}.
\end{equation*}
We observe here that the action of the conormal derivative operator at the point $x\in\pO$ depends only on the values of the Lam\'e coefficients at this point and thus is the same for constant coefficients and for smooth variable coefficients.

For the case of constant coefficients, which we consider further on in this section, the NP operator is defined as
\begin{equation*}%\label{NP}
    (\Kb f)(x)=(P.V.) \int_{\pO} \pp_{\n_x}\G(x-y)\f(y)d\s(y), \, x\in\pO,
\end{equation*}
where the integral is understood in the principal value sense. It is known that, for a smooth boundary, this is a bounded operator in all Sobolev spaces, however it is not compact, unlike its scalar (say, electrostatic) counterpart.

The Neumann-Poincar\'e operator $\Kb$ is related to the boundary problem in the following way. Consider the Neumann boundary problem in $\O$ with the prescribed stress $g\in C^{\infty} $ on $\pO$:
\begin{equation}\label{BVP}
    \Lc_{\l,\m}u=0, x\in\O;\, \pp_{\n_x}u(x)=g(x), \, x\in\pO.
\end{equation}
Then the solution may be sought for as the \emph{single layer} potential of a certain unknown  vector-valued function $\f(x)$ on $\pO$. \emph{For constant coefficients},
this means that
\begin{equation}\label{SLP}
    u(x)=\Sbb[\f](x)=\int_{\pO}\G(x-y)\f(y)d\s(y), \, x\in\O.
\end{equation}
The function $u(x)$  in \eqref{SLP} satisfies the equation $\Lc u=0$ in $\O$. The limit from inside of the conormal derivative of the single layer potential  is  expressed via the NP operator
 by the jump relation
 \begin{equation*}%\label{jump}
    \pp_{\n_x} \Sbb[\f](x)=(-\frac12 \Ib+\Kb)[\f](x),\, \mathrm{a.e. }\, x\in\pO.
 \end{equation*}
 Thus, to find a solution of the boundary problem \eqref{BVP}, it suffices to solve the integral equation
 \begin{equation*}%\label{IntEq}
    (-\frac12\Ib+\Kb)[\f](x)=g(x),\,  x\in \pO.
 \end{equation*}
The operator $\Kb$ is known to be noncompact, actually, a singular integral operator, see, e.g., \cite{Kupr79}, and in \cite{3D}, \cite{2D} it was established that this operator is polynomially compact,
\begin{equation*}%\label{polComp}
    \Kb(\Kb^2-\kb_0^2)\in\SF_\infty, d=3;\, (\Kb^2-\kb_0^2)\in\SF_\infty, d=2,
\end{equation*}
where $\SF_\infty$ denotes the ideal of compact operators in $L^2(\pO)$ and $\kb_0=\frac{\m}{2(\l+2\m)}$. Moreover, it was shown in \cite{3D} that these three points $0, \pm \kb_0$ are, in fact, points of essential spectrum of $\Kb$ and the discrete eigenvalues of $\Kb$ may   converge only to these three points.
This fact has been recently proved even for $C^{1, \alpha}$ ($\alpha>0$) surfaces \cite{KK}.
The reasoning in  \cite{3D} was based upon finding a local, and then global, representation of $\Kb$ as a zero order  pseudodifferential operator, up to a weaker error term, and then deriving an expression for the principal symbol for this operator.
In a similar way, earlier, in \cite{2D}, the NP operator has been studied for the 2-dimensional elasticity system. There, the essential spectrum of $\Kb$ consists of two points, $\pm \kb_0$, and the polynomial compactness, $(\Kb^2-\kb_0^2)\in \SF_{\infty}$ takes place.
\subsection{An alternative solution}\label{subAlter}
When presenting their considerations, the authors of \cite{3D} were unwary that the pseudodifferential representation of the NP operator for the homogeneous elasticity problem had been found almost 20 years earlier by M. Agranovich, B. Amosov and M. Levitin in \cite{AgrLame}. We explain here the reasoning in \cite{AgrLame} (in their notations). We do this, rather than simply referring to suitable pages in \cite{AgrLame}, since we will need it to compare with the reasoning when treating the variable coefficients case.

The NP operator, denoted  by $\Kb$ above, called in \cite{AgrLame} 'the direct value of the derivative of the single layer operator', is denoted \emph{there} by $B'$, and its $L_2(\pO)-$ adjoint, 'the direct value of the double layer potential' is therefore denoted by $B$. The main result in \cite{AgrLame} we need here is the representation of $B, B'$ as pseudodifferential operators.  The reasoning in this task goes in  the following way.

The Lam\'e operator $\Lc$ is a formally self-adjoint homogeneous second order \emph{constant coefficients } differential operator with symbol
\begin{equation*}%\label{LameSymbol}
    \ell(\x)=-|\x|^2(\m E+(\l+\m)\L(\x)),
\end{equation*}
where $E$ is the identity $3\times 3$ matrix and the entries of $\L(\x)$ are $\x_j\x_k|\x|^{-2}$. It is easy to check that
\begin{equation*}%\label{Lambda idempotent}
    \L^2(\x)=\L(\x),
\end{equation*}
and therefore,  under the convexity conditions, the symbol $\ell(\x)$ has the inverse, the matrix
\begin{equation}\label{Lambda inverse}
    \ell^{-1}(\x)=-\frac{1}{\m|\x|^2}\left( E-\frac{\l+\m}{\l+2\m}\L(\x)\right).
\end{equation}
This confirms the ellipticity of the Lam\'e operator. The formal self-adjointness is obvious. Moreover, the symbol $-\ell(\x)$ is a positive definite matrix since
\begin{equation*}%\label{PositDefin}
    -\langle\ell(\x)v,v\rangle=\m|\x|^2|v|^2+(\l+\m)|\langle\x, v\rangle|^2,
\end{equation*}
where $\langle .,.\rangle$ denotes the scalar product in the 3-dimensional complex space $\C^3$. Thus, the operator $-\Lc$ is strongly elliptic.

The fundamental solution of the Lam\'e equation is the inverse Fourier transform of the symbol $\ell(\x)^{-1}$, which has the leading term \eqref{Lambda inverse}.
 By directly checking the Lopatinsky
conditions, we can see that the Dirichlet and Neumann problems for the Lam\'e system are elliptic.

To obtain a pseudodifferential representation of the potential operators $B, B'$, a special frame and a local co-ordinate system are fixed. Near a point $x^0\in \pO$, the origin of the co-ordinate system  is placed at $x^0$, the co-ordinate $x_3$ takes the direction of the exterior normal to $\pO$ at $x^0$ and $x'=(x_1,x_2)$ lies in the tangent plane $T_{x^0}\pO$ to $\pO$ at $x^0$. Correspondingly, the frame in which the Lam\'e equations are expressed is located along these co-ordinate axes. Further on, let $x_3=X(x')$, $x'\in T_{x^0}\pO$, $x'=(x_1,x_2)$ be the equation of $\pO$ near $x^0$. The new co-ordinates $\tilde{x}_{j}, j=1,2,3,$ are introduced, rectifying $\pO$:
\begin{equation}\label{CoordRect}
    \tilde{x}_1=x_1, \,\tilde{x}_2=x_2, \tilde{x}_3= x_3-X(x').
\end{equation}
We omit further on the tilde sign, for simplicity of notation.
In these local co-ordinates, the operator $\Lc^{-1}$ ceases to be a pseudodifferential operator with symbol not depending on $x$, but, with accordance to the rules of variables change, becomes an operator with a variable symbol, however the principal symbol at the point $x_0$ remains the same since the Jacobi matrix of the transformation \eqref{CoordRect}  at this point is the unit matrix. Thus, the operator $\Lc^{-1}$ is, locally, near $x^0$ a pseudodifferential operator with principal symbol
\begin{equation}\label{symbol transformed}
\rb_0(x,\x)=\ell(\x)^{-1}(1+O(x')).
\end{equation}

The operator $A$, the direct value of the single layer potential, can be written as the composition of three operators. First, a function $f$ on $\pO$ extends by zero onto all of $\R^3$, obtaining the distribution $\Cc[f](x)$. Then the pseudodifferential operator $\Lc^{-1}$ is applied, in the sense of distributions, to $\Cc[f](x)$. Finally, the resulting distribution, being, in fact, a function, is restricted back to $\pO$. The  principal symbol $\s_{A}$ of this operator on $\pO$ can be calculated using Proposition 3.5 in \cite{AgrAmos}. We need to do this only at the point $x^0$, i.e., $(0,0,0)$. According to formula (3.8) there,
\begin{equation*}%\label{sigma(A)}
    \s_{A}(x', \x')=\frac{1}{2\pi}\int_{-\infty}^{\infty} \rb_0(x',0;\x',\x_3)d\x_3, \x'\ne 0.
\end{equation*}
We substitute here the expression \eqref{symbol transformed} for the symbol $\rb_0(x',0;\x',\x_3)$ to obtain
\begin{equation*}%\label{sA}
    \s_{A}(0; \x')=\frac{1}{2\pi\m}\int_{-\infty}^{\infty} \left(\frac {\l+\m}{\l+2\m}|\x|^{-2}\L(\x)-|\x|^{-2}E\right)d\x_3, \, \x'=(\x_1,\x_2)\ne 0.
\end{equation*}
Using the easily verified relations
\begin{equation*}
    \int_{-\infty}^{\infty}|\x|^{-2}d\x_3=\frac{\pi}{|\x'|}, \,\int_{-\infty}^{\infty}\frac{\x_3^2}{|\x|^4}d\x_3=\frac{\pi}{2|\x'|}, \,\int_{-\infty}^{\infty}\frac{d\x_3}{|\x|^4}=\frac{\pi}{2|\x'|^3},
\end{equation*}
we obtain
\begin{equation*}%\label{SymbA}
    \s_{A}(0,\x')=\frac{1}{\m|\x'|}\left(\frac{\l+\m}{\l+2\m}\left(\L(\x')-E\right)\right).
\end{equation*}

Now the principal symbol of $B$ is calculated.
It is shown in \cite{Kupr79}, Chapter 2, \S 4,  that
\begin{gather*}%\label{TGamma}
    (\pp_{\n_x}\G(x))_{jk}=\m(\l'-\m')\frac{\n_j(x)-\n_k(x)}{|x|^3}+\\
    \nonumber \left(\m(\m'-\l')\d_{jk}-6\m\m'\frac{x_jx_k}{|x|^2}\right)\sum_{l=1}^{3}\n_l(x)\frac{x_l}{|x|^3},
\end{gather*}
where $\m'=\frac{\l+3\m}{4\pi\m(\l+2\m)}$, $\l'=\frac{\l+\m}{4\pi\m(\l+2\m)}$.
Since the normal vector is orthogonal to the tangential ones, we have
\begin{equation*}%\label{orthog}
    \sum\n_l(y)\frac{x_l-y_l}{|x-y|}\le C|x-y|.
\end{equation*}
Therefore, the leading part of singularity of the  kernel of the operator $B$ on the surface $\pO$ has the entries
\begin{equation*}%\label{SingB}
    \m(\l'-\m')\frac{\n_k(y)(x_j-y_j)-\n_j(y)(x_k-y_k)}{2|x-y|^3}.
\end{equation*}
Here, by the definition of $\l',\m'$,
\begin{equation*}%\label{ml}
    \l'-\m'=\frac{1}{2\pi(\l+2\m)}>0.
\end{equation*}
We determine the principal symbol of the operator $B$, using the same co-ordinate system as above, when calculating the symbol of the operator $A.$ We start with the pseudodifferential operator with the symbol $|\x'|^{-1}$. This is the  integral operator with kernel $(2\pi|x'-y'|)^{-1}$. Next, note that $\frac{x_j}{|x'|^3}=-\partial_j(|x|^{-1})$ and we know that the differentiation of the kernel in $x_j$ corresponds to the  multiplication of the symbol by $i\x_j$. Therefore  the expression for the principal symbol of $B$ is
\begin{equation}\label{PrincBconst}
    \s_B(\x')=\frac{\pi\m(\l'-\m')i}{|\x'|}\left(
                                             \begin{array}{ccc}
                                               0 & 0 & -\x_1 \\
                                               0 & 0 & -\x_2 \\
                                               \x_1 & \x_2 & 0 \\
                                             \end{array}
                                           \right).
\end{equation}
The matrix in \eqref{PrincBconst} is self-adjoint, and therefore the principal symbol of the NP operator $\Kb=B'$ is the same, $\s_{B'}=\s_B$.
The eigenvalues of the symbol \eqref{PrincBconst} are: $0, \pm \kb_0 $, where $\kb_0=\pi\m(\l'-\m')=\frac{\m}{2(2\m+\l)}$.
Now let us recall that the essential spectrum of an operator $\Kb$ in a Hilbert space consists of such complex numbers $\z$, for which the operator $\Kb-\z$ is not Fredholm. For classical pseudodifferential operators, the Fredholm property is equivalent to ellipticity (see, e.g., \cite{Agr Obzor}). Thus, the operator $B-\z$ is not Fredholm exactly for those numbers $\z$ that are contained in the range of the eigenvalue branches of the symbol \eqref{PrincBconst} since it is exactly at these points $\z$ that the matrix $\s_B(\x')-\z$ is not invertible. There are exactly three such numbers: $0, \pm\kb_0$. In this way, the main result of the paper \cite{3D} is reproduced.
\begin{theorem}\label{EssSpecConst}The essential spectrum of the NP operator $\Kb$ consists of tree points, $0, \pm\kb_0$.
\end{theorem}
The corollary, stating that $\Kb$ is polynomially compact, namely, that $\Kb(\Kb^2-\kb_0^2)\in\SF_\infty$ follows directly by calculating the principal symbol of this operator and using the spectral mapping theorem.

 \begin{remark}\label{RemFS}The analysis of the reasoning above demonstrates which properties of the fundamental solution have been, actually, used in this proof. Let us replace the kernel $\G(x-y)$ in the definition of our potentials,  by some other function $\bG(x,y)$, $x,y\in \overline{\O}$, smooth in its variables for $x\ne y$, up to the boundary $\pO$, and possessing, for smooth variable $\m(x),\l(x),$ the same singularity  as $x\to y$:
\begin{equation*}%\label{Gamma}
    \bG(x,y)=\G_{\l(x),\m(x)}(x-y)+R(x,y)
\end{equation*}
where $R(x,y)$ (automatically, smooth for $y\ne x$) is subject to
\begin{equation*}%\label{CondR}
    R(x,y)=O(1);\, \nabla_x R(x,y), \nabla_y R(x,y)=O(|x-y|^{-1}).
\end{equation*}
Then, in the construction of the single layer and double layer potentials, the representation of the principal symbol at a point at the boundary involves exclusively the values of $\m$ and $\l$ at this point.
This circumstance will be utilized further on.

We would note also that the fundamental solution $\G(x,y)$ is not unique. In particular, if $R(x,y)$ is a function defined for $x,y$ in a certain neighborhood $\O'$ of $\overline{\O}$, symmetric, $R(x,y)=R(y,x)^*$ and satisfying in $\O'$ the homogeneous Lam\'{e} equation in $x$ variable (and, by symmetricity, in $y$ variable), the potentials constructed with the kernel $\G+R$ instead of $\G$ would possess the same properties, such as being the solution of the Lam\'{e} equations in $\O'$ and jump relations on $\pO$, as the usual potentials. The only difference is that the $\G+R$- potentials do not satisfy the conditions at infinity (they  may even turn out to be not defined far away from $\O$), but this latter property is not needed in our considerations.\end{remark}

\section{Single layer and double layer potentials for a nonhomogeneous body}\label{SectVarParameters}

\subsection{The fundamental solution}
From now on, we suppose that the Lam\'{e} coefficients $\m(x)$, $\l(x)$ are functions of $x\in\overline{\O}$, smooth up to the boundary, with the convexity conditions satisfied uniformly in $\overline{\O}$. We consider the body $\O$ as a subset of the three-dimensional torus $\T^3$. The coefficients $\m(x)$, $\l(x)$ can be continued as smooth functions on $\T^3$ with the convexity conditions \eqref{ConvexityU} satisfied uniformly in $\T^3$. The elasticity operator $-\Lc$ can be understood as a second order self-adjoint  strongly elliptic system on $\T^3$. Moreover, it is non-negative in the sense of $L_2(\T^3)$

As follows from the general theory of elliptic operators (see, e.g., \cite{Agr Obzor}), the operator $\Lc$ has discrete spectrum, in particular the zero subspace $E_0$ is finite dimensional. It is the space where the quadratic form
\begin{equation*}%\label{Qform}
    (-\Lc u,u) =\int_{\T^3}(\m(x)|\nabla u|^2+(\m+\l)|\div u|^2)dx
\end{equation*}
vanishes.  On the orthogonal complement to this subspace the operator $-\Lc$ is positively definite.

We fix a smooth nonnegative function $h(x)$ such that $h(x)\equiv 0$ in $\O$ and it is positive on some open subset in  $\O'\subset \T^3\setminus \overline{\O}$. Therefore, the quadratic form $-(\Lc u,u)+\int h(x)|u|^2 dx$ is strictly positive and the operator $-\Lc^h\equiv-\Lc+h$ is invertible in $L_2(\T^3)$. Note that on the domain $\O$ the differential operators $\Lc$ and $\Lc^h$ act in the same way. We denote by $\Rc^h $ the inverse operator for $\Lc^h$. It is a bounded order $-2$ pseudodifferential operator
 on $\T^3$. Its principal symbol equals the inverse of the principal symbol of $\Lc$, i.e., is given by the proper generalization of \eqref{Lambda inverse}:
 \begin{equation}\label{PrincSymbVar}
    \s_{\Rc^h}= -\frac{1}{\m(x)|\x|^2}\left( E-\frac{\l(x)+\m(x)}{\l(x)+2\m(x)}\L(\x)\right).
 \end{equation}
 We note here that due to the nice topology of the torus, the expression \eqref{PrincSymbVar} makes sense globally on $\T^3$.

 Applying the Fourier transform, or, what is more convenient, the Fourier series, we can express $\Rc^h$ as a self-adjoint  integral operator with weakly polar polyhomogeneous kernel,
 \begin{equation*}%\label{Resolv}
    \Rc^h [u](x) =\int_{\T^3}\bG(x,y)u(y)dy,
 \end{equation*}
 with leading singularity as $x\to y$
having order $-1$.

By the construction, the kernel $\bG(x,y)$ satisfies the equation $\Lc^h_x \bG(x,y)=0$ for $x,y\in \overline{\O}$ as soon as $x\ne y$, as well as the adjoint equation $\Lc^h_y \bG(x,y)=0$. Since $h=0$ in $\overline{\O},$  $\bG(x,y)$ satisfies the equation $\Lc_x \bG(x,y)=0$, $\Lc_y \bG(x,y)=0$ for $x,y\in \overline{\O}$,  $x\ne y.$ We will call $\bG(x,y)$ the fundamental solution of the Lam\'{e} system (this kernel, of course depends on the choice of the function $h$ but we do not need to mark this dependence in our notations, see Remark \ref{RemFS}.).

Being the integral kernel of a  classical order $-2$ pseudodifferential operator on $\T^3$, the kernel $\bG(x,y)$ admits a representation
\begin{equation*}%\label{kernel}
    \bG(x,y)\sim \bG_{-1}(x,x-y) +\widetilde{\bG}(x,y),
\end{equation*}
where $\bG_{-1}(x,x-y)$ is an integral  kernel, homogeneous  of order $-1$ in $x-y$, and $\widetilde{\bG}(x,x-y)$ is a bounded kernel, satisfying $\nabla_y \widetilde{\bG}(x,x-y)=O(|x-y|^{-1})$. Since the principal symbol of the operator $\Rc^h$ at the point $x\in\T^3$ depends  on the value of the Lam\'{e} coefficients at this point only, the same is correct for the kernel $\bG_{-1}$. Therefore, $\bG_{-1}(x,x-y)$ is given by the Kelvin matrix calculated for the values of $\m,\l$ at the point $x$.

\subsection{The layer potentials and operators} The approach we use to construct the single and double layer potentials follows mostly the general considerations presented  in \cite{McLean}, Ch 6., \cite{Agrgen}, Ch.12, and \cite{Duduchava}.

The single layer potential of a distribution-density $\psi\in H^{-\frac12}(\pO)$ is defined in the following way. On the space of smooth functions on $\T^3$ we consider the mapping $\g$ associating with a  function its restriction to $\pO$. By the
embedding and trace theorems, $\g$ extends to a bounded operator in Sobolev spaces, $\g: H^{1-s}(\T^3)\to H^{-s+\frac12}(\pO)$, $s<\frac12$. The adjoint operator $\g^*$, the one of extension of the distribution on $\pO$ to $\T^3$ by zero, is a bounded operator from $H^{s-\frac12}(\pO)$ to $H^{s-1}(\T^3)$.
 Now we apply the pseudodifferential operator $\Rc^h$ to $\g^*\psi$:
 \begin{equation*}%\label{SingLayer}
    \Ac[\psi]=\Rc^h[\g^* \psi].
 \end{equation*}
 It is easy to verify (see \cite{McLean}, Sect. 6, or \cite{Agrgen}, Sect 12.2, for details) that for $\psi\in H^{s-\frac12}(\pO)$, $\Ac[\psi]$ is a function in $H^{1+s}(\T^3),$ satisfying the differential equation $\Lc^h \Ac[\psi]= \g^* \psi$ in the sense of distributions and, in particular, $\Lc \Ac[\psi]=0$ in $\O$. The restriction of $\Ac[\psi]$ to $\pO$ makes therefore sense and thus the direct value of the single layer potential is defined as
 \begin{equation*}
    A[\psi]=\g \Ac[\psi].
 \end{equation*}
 The operator $A$ is an order $-1$ positive pseudodifferential operator on $\pO$. We denote by $\Ac(x,\x)$ its leading symbol, which at a point $x\in\pO$ depends only on the values of the Lam\'e coefficients $\l$,$\m$ at the point $x.$ As any order $-1$ pseudodifferential operator on a two-dimensional surface, the operator $A$ is an integral operator, with  polyhomogeneous kernel, the leading singularity being of order $-1$. We denote the corresponding part of the kernel by $\aF(x,y,x-y)$, the function $\aF$ being smooth in the first two variables and order $-1$ homogeneous in the third one.

 Now we pass to defining the double layer potential. For a strongly elliptic Lam\'{e} system, under our smoothness and ellipticity conditions, the operator of the conormal derivative \eqref{symgrad}, written as
 \begin{equation*}%\label{conormalGen}
    \pp[u](x)=\sum\n_j(x)a_{j,k}(x)\g\partial_k u(x),
 \end{equation*}
is well defined for functions $u\in H^{2-s}(\T^3)$, $s<\frac12$. (Here $a_{jk}=a_{jk}(x)$ are the matrix coefficients in the Lam\'{e} system.) In fact, the differentiation and then the passage to the boundary decrease the index $2-s$ by $\frac32$ and traces from both sides coincide; so, $\pp$ is a bounded operator from $H^{2-s}(\T^3)$ to $H^{\frac12-s}(\pO).$
The adjoint operator $\pp^*$, defined by
\begin{equation*}%\label{adj}
    (\pp^*[\f],u)_{\T^3}=(\f, \pp[u])_{\pO}, ,
\end{equation*}
acts from $H^{s-\frac12}(\pO)$ to $H^{2-s}(\T^3)$.
We define the double layer operator $\Bc$ by
\begin{equation}\label{DoubLayerGen}
    \Bc[\f]=\Rc^h \pp^*[\f].
\end{equation}
As explained in Section 12.2 in \cite{Agrgen}, this definition, for $\f\in H^{\frac12}(\pO)$, coincides with the more intuitive usual  definition
\begin{equation}\label{DoubleLayerIntuitive}
    \Bc[\f](x)=\int_{\pO} \pp_{\n_y}\bG(x,x-y)\f(y) dS(y),\, x\in\O,
\end{equation}
where, recall, $\bG(x,x-y)$ is the integral kernel of the resolvent operator $\Rc^h=(\Lc^h)^{-1}$.

The integral operator with kernel $\pp_{\n_y} \bG(x,x-y)$, acting from the boundary to the domain, is adjoint to the operator with kernel $\bG(x,x-y) \pp_{\n_x}^*$, acting from the domain to the boundary. In fact,
for a real-valued smooth test function $g(x)$ with support lying inside $\O$ or outside $\overline{\O}$, we have
\begin{gather*}%\label{adjoint}
 \iint\limits_{\T^3\times\T^3}\bG(x,x-y)\pp_{\n_x}^*[\f](y)g(x)dxdy=
 \int\limits_{\T^3} \pp_{\n_y}[\f](y)\int\limits_{\T^3}\bG(x,x-y)g(x)dx dy=\\ \nonumber \int\limits_{\pO} \f(y)\pp_{\n_y}\int\limits_{\T^3}\bG(x,x-y)g(x)dx dS_y=\int\limits_{\T^3}\int_{\pO}(\pp_{\n_y}\bG(x,x-y))[\f](y)g(x)dS_ydx.
\end{gather*}

 The function $\Bc[\f](x)$, due to the construction of the kernel $\bG(x,x-y)$, satisfies in $\O$ the equation $\Lc \Bc[\f](x)=0$. What we need now is to study the behavior of the potential \eqref{DoubLayerGen} at the boundary $\pO$. This
behavior has been investigated in detail for general strongly elliptic second order system, see \cite{McLean}, Ch.6, \cite{Agrgen}, Ch.12 (a construction for a strongly elliptic system of an arbitrary order is described in \cite{Seeley}). In our case, the reasoning can be made more transparent, with an explicit calculation of the principal symbol avoiding the technicalities.

We consider the behavior of the potential \eqref{DoubleLayerIntuitive} as the point $x$ approaches $x^0\in \pO$ from inside. The operator $\Bc$ can be represented as the sum
$\Bc=\Bc_{-1}+\tilde{\Bc}$, where $\Bc_{-1}$ is the integral operator with kernel $\pp_y\bG_{-1}(x,x-y)$ and $\widetilde{\Bc}$ is the remainder, the operator with kernel $\pp_y\widetilde{\bG}(x,y)$. The latter kernel has singularity of order not greater then $-1$ and therefore is  continuous up to the boundary (for a continuous function $\f$). Therefore, the interior limit value of $\widetilde{\Bc}[\f]$ on $\pO$ is the direct value of the potential, $\int_{\pO}\pp_y\tilde{\bG}(x_0,y)\f(y) dS(y)$, $x_0\in \pO.$ Being a kernel with singularity of order not higher than $-1$, this limit operator is a pseudodifferential operator of order not greater than $-1$ on $\pO$.

As for the leading term, we can, for $x\to x^0$, represent $\bG_{-1}(x,x-y)$ as
\begin{equation}\label{splitting}
\bG_{-1}(x^0,x^0-y)+ (\bG_{-1}(x,x-y)-\bG_{-1}(x^0,x^0-y)).
\end{equation}
 The second term in \eqref{splitting} defines an operator $\Bc'$, and for a continuous function $\f$, $\Bc'[\f](x)\to 0$ as $x\to x^0$.  As for the first term in \eqref{splitting}, it is nothing but the double layer potential for the Lam\'{e} operator with constant coefficients, actually, $\m,\l$, frozen at the point $x^0\in \pO$.
For this latter potential, the limit value is known; it is described in Section \ref{SecLameConst}:
\begin{equation*}%\label{LimVariable}
    \lim_{x\to x^0}\Bc[\f](x)=\frac12 \f(x^0)+ B[\f](x^0),
\end{equation*}
where the NP operator, in other words, the direct value of the double layer potential, $x^0\in \pO$ is defined as:
\begin{equation}\label{direct value}
B [\psi] (x^0)= (P.V)\int_{\pO}\pp_y\bG(x^0,x^0-y)\f(y) dS(y),\, .
\end{equation}

Thus, the operator $B$ defined in \eqref{direct value} is natural to consider as the generalization, to the case of a nonhomogeneous media, of the classical NP operator, since this operator plays the identical role and can be used in the usual way for the study of  boundary problems.
The operator $B$ consists of the leading term $B_0$, with kernel $\pp_{\n_y}\bG_{-1}(x_0,x_0-y)$,
and a weaker
one, a pseudo-differential operator on $\pO$ of order $-1$, therefore, a compact one. As  for the leading term, for a fixed $x^0\in \pO$,
it is identical with the NP operator for the homogeneous medium, with Lam\'{e} coefficients frozen at the point $x^0.$ The principal symbol for this operator was already calculated, see Sect. \ref{SecLameConst}. In the same local co-ordinates and frame, it equals

 \begin{equation}\label{PrincBNonconst}
    \s_B(x^0,\x')=\frac{\pi\m(x^0)(\l'(x^0)-\m'(x^0))i}{|\x'|}\left(
                                             \begin{array}{ccc}
                                               0 & 0 & -\x_1 \\
                                               0 & 0 & -\x_2 \\
                                               \x_1 & \x_2 & 0 \\
                                             \end{array}
                                           \right).
\end{equation}

\subsection{The essential spectrum of the NP operator and of the modified NP operator.}

The eigenvalues of the symbol \eqref{PrincBNonconst} are calculated easily; they are equal to
$0,\pm\kb_0(x^0)$, where $\kb_0(x^0)=\frac{\m(x^0)}{2(2\m(x^0)+\l(x^0))}$ for a point $x^0$ at the boundary. It is well known that for a matrix pseudodifferential  operator of order zero on the manifold $\pO$, the essential spectrum consists of the range of eigenvalues of the principal symbol, considered as functions on the cosphere bundle of $\pO$ (just observe, if $\z$ does not belong to $\{0\}\cup\Ran(\pm\kb_0)$, the symbol $\s_{B}-\z$ is invertible for all $(x^0,\x)\in T^*(\pO)$ and a pseudodifferential operator with principal symbol $(\s_{B}-\z)^{-1}$ is a regularizer for $B-\z$; on the other hand, if $\kb_0-\z, \z\ne0,$ vanishes somewhere, the operator $B-\z$ is not elliptic and, therefore, not Fredholm). This gives us our main result on the essential spectrum.

\begin{theorem}\label{mainTheorem}Let $\O$ be an elastic isotropic body with Lam\'e  coefficients $\m(x)$, $\l(x)$, smooth up to the smooth boundary. Let $\pO_1,\dots,\pO_N$ be the connected components of the boundary $\pO$. For each component $\pO_l$, $l=1,\dots,N$ denote by $\pm\D_l$ the (of course, closed) range of the function $\kb_0(x)=\pm\frac{\m(x)}{2(2\m(x)+\l(x))}$, $x\in \pO_l$. Then the essential spectrum of the NP operator coincides with the union of the point $0$ and the union of the intervals $\pm\D_l$, $l=1,\dots,N$.
\end{theorem}

To illustrate this result, consider some special cases. Suppose that $\pO$ consists of several connected components, $\pO_l, \, l=1,\dots,N$. If the Lam\'{e} parameters $\m(x),\l(x),$ probably, variable, satisfy,
\begin{equation}\label{locconst}
    \frac{\m(x)}{2(2\m(x)+\l(x))}=\kb_l, \, x\in\pO_l, l=1,\dots,N,
\end{equation}
so that $\kb_0$ is constant on each component of the boundary, then the intervals $\pm\D_l$ degenerate to single points and the essential spectrum of the NP operator consists of the points $0, \pm\kb_l$, i.e., of an odd number of points. In this case, the $NP $ operator $B$ is polynomially compact,
\begin{equation*}%\label{PolComp}
    p(B)\in\SF_\infty,\, p(s)=s\prod(s^2-\kb_l^2).
\end{equation*}
One can see that this is the only case when this polynomial compactness takes place. If, on the contrary, the function $\kb_0(x)$ is nonconstant on the components $\pO_l,$ for all $l$, the range of $\kb_0$ on $\pO_l$ is a nondegenerate closed interval $\D_l\subset (0,\infty)$. Therefore, the intervals $\pm \D_l$ compose the continuous spectrum of $B$. It is also possible that on some components of the boundary the function $\kb_0$ is constant while it is not constant on other components. In this general case the essential spectrum of $B$ consists of an even collection of symmetric intervals $\pm \D_l$, the point zero, as well  as, possibly, of an even set of isolated points.

We can introduce the notion of a \emph{modified NP operator} so that, even in the nonhomogeneous case, the polynomial compactness takes place.

\begin{definition} Let $B$ be the NP operator on $\pO$ for a smooth (variable) elastic media. The modified NP operator $\tilde{B}$ is defined as
\begin{equation}\label{modif}
    \widetilde{B}=\kb_0(x)^{-1}B, \, x\in \pO.
\end{equation}
\end{definition}
\begin{theorem}\label{modifTh}
The essential spectrum of the modified NP operator \eqref{modif} consists of three points, $0, \pm 1$. The operator is polynomially compact: $p(\widetilde{B})\in\SF_{\infty}, p(t)=t(t^2-1).$
\end{theorem}

\begin{proof}The principal symbol of the  zero order operator $\tilde{B}$ equals $\kb_0(x)^{-1}\s_B(x,\x'),$ so it is the matrix
\begin{equation*}
  i|\x'|^{-1}  \left(
                                             \begin{array}{ccc}
                                               0 & 0 & -\x_1 \\
                                               0 & 0 & -\x_2 \\
                                               \x_1 & \x_2 & 0 \\
                                             \end{array}
                                           \right).
\end{equation*}
This matrix has eigenvalues $0,\pm 1$, so we return to the situation of a homogeneous media.
\end{proof}

\section{Eigenvalues convergence rate}
\subsection{General eigenvalues estimates.} We have found the essential spectrum of the operator $\Kb$. This operator may also  have eigenvalues. Some of them may lie in the gaps of the essential spectrum  (not excluding the semiinfinite ones); it is also possible that some eigenvalues are embedded in the essential spectrum.  For the eigenvalues  in the gaps, the  only limit points may be the points of the essential spectrum. It turns out that it is possible to evaluate the rate of convergence of the eigenvalues to these points, provided the boundary $\pO$ is infinitely smooth. We consider the case of a connected boundary  $\pO$ and a homogeneous body first.
In this case, due to our Theorem \ref{mainTheorem}, the essential spectrum consists of three points, $0, \pm \kb_0$. We denote the eigenvalues in some neighborhoods of these points by $\lambda_j^{0,\pm},\lambda_j^{+,\pm},\lambda_j^{-,\pm}$, where, for example, $\lambda_j^{-,+}$ denotes the sequence of eigenvalues approaching $-\kb_0$ from above, with similar meaning for other notations.
\begin{theorem}\label{Th.Eigenvalues} Let $\O$ be a bounded domain with smooth connected boundary $\pO$. Then  for the eigenvalues of the operator $\Kb$ the following estimates hold:
\begin{equation}\label{rate of convergence}
    \lambda_j^{0,\pm}=O(j^{-\frac12}), \, |\lambda_j^{\pm,\pm}-(\pm\kb_0)|=O(j^{-\frac12}).
\end{equation}
Of course, if one of the above sequences of eigenvalues is finite or void, the corresponding estimate holds by triviality.
\end{theorem}
\begin{proof} We consider the case of $\lambda_j^{+-}$; the other five cases are dealt with in the same way.

First, we consider the operator $\Kc=A^{\frac12}\Kb A^{-\frac12}$, where $A$ is the single layer potential operator. By the Plemeli formula, this operator is selfadjoint in $L^2(\pO)$ and, being similar to $\Kb$, it has the same spectrum. We choose a point $\vs$ in the interval $(\kb_0/2,\kb_0)$ such that $\vs$ is not an eigenvalue of $\Kb$ and consider the operator

\begin{equation}\label{convergence}
\Mb=m(\Kc); m(s)=(\kb_0-\vs)^{-1}-(s-\vs)^{-1}.
\end{equation}
By construction and due to the spectral mapping theorem, the operator $\Mb$ has  essential spectrum at the points $0, (\kb_0-\vs)^{-1}+\vs^{-1}, (\kb_0-\vs)^{-1}+(\kb_0+\vs)^{-1}$, so the point $0 =m(\kb_0)$ is now the lowest point of the essential spectrum of $\Mb$ and the eigenvalues of $\Kb$ in a lower neighborhood of $\kb_0$, i.e., exactly $\l_j^{+-}$, with exception of a finite number of them,  are mapped   into the negative eigenvalues of $\Mb$. The operator $\Mb$ is a zero order pseudodifferential  operator with principal symbol $\s_{\Mb}(x,\x')$ which equals $m(\s_{B(x,\x')})$, where $\s_B$ is given in \eqref{PrincBNonconst}. In particular, this principal symbol is a nonnegative matrix.

Now we recall the sharp G${\aa}$rding inequality, see \cite{LaxNir}, stating that for a zero order matrix pseudodifferential operator with nonnegative principal symbol,
\begin{equation*}%\label{GardingAppl}
    (\Mb u,u)\ge -C\|u\|_{H^{-1/2}}^2,
\end{equation*}
This inequality was established in \cite{LaxNir} for operators in the Euclidean space, but it carries over automatically to a smooth compact manifold by means of the usual localization.
Now, by the variational principle, the number of negative eigenvalues of $\Mb$, smaller than $-t$, $t>0$, is majorated by the number of negative eigenvalues smaller than $-t$ of the operator in $L^2$ defined by the quadratic form $-C\|u\|_{H^{-1/2}}^2=-C(\Gb u,u)$, where  $\Gb=(-\D+1)^{-\frac12}$ and $\D$ is the Laplacian on $\pO$. This operator $\Gb$ is a pseudodifferential operator of order $-1$, and therefore, for its eigenvalues, the estimate $\l_j(\Gb)=O(j^{-\frac12})$ holds. Finally, the latter estimate, by means of the spectral mapping theorem, is transformed into the required estimate for eigenvalues of $\Kb$.

Other five sequences of eigenvalues of $\Kb$ are treated in the same way, just with a proper choice of the function $m(s)$ in \eqref{convergence}, so that the concrete point of the essential spectrum under study is transformed into zero and the eigenvalues of interest are transformed into the negative ones.
\end{proof}

The above result can be easily extended to the case of a nonhomogeneous body. Let, again, $\O$ be a body with smooth boundary $\pO$ and $\m(x),\l(x)$ be smooth Lam\'{e} parameters. Consider, again, the case of connected $\pO$, for simplicity of notations. Then, if the parameters are constant on the boundary, the result and the reasoning in Theorem \ref{Th.Eigenvalues} remain true without changes. Suppose now that the Lam\'{e} parameters are nonconstant on the boundary, so, by Theorem \ref{mainTheorem},  the essential spectrum  consists of the point $0$ and two intervals, $[-b,-a]$ and $[a,b]$. We denote by $\l_j^{a,\pm}$ the eigenvalues approaching $\pm a$ from outside the intervals of the essential spectrum; similarly, the eigenvalues  $\l_j^{b,\pm}$ are defined. The notation $\l_j^{0,\pm}$, used previously, is preserved.

\begin{theorem}\label{Th.Eigenv.Nonhm} For the eigenvalues of the NP operator for a nonhomogeneous elastic body, the estimates
\begin{equation}\label{estimNonhom}
    \l_j^{a,\pm}\mp a, \l_j^{b,\pm}\mp b, \l_j^{0,\pm}=O(j^{-\frac12}).
\end{equation}
\end{theorem}
 The proof of this  theorem goes similarly to the one of Theorem \ref{Th.Eigenvalues}. We study, for example,  $\l_j^{a,+}$, i.e., the eigenvalues converging to $a$ from below.  Consider the rational function  $m(s)=(a-\vs)^{-1}-(s-\vs)^{-1}$; in this way  the point $a$ of the essential spectrum, under study, is mapped into  zero, and all other points of the essential spectrum are mapped into positive numbers, with  $\vs$  chosen arbitrarily in $(a/2,a)$, being however  not an eigenvalue of $\Kb$. The principal symbol of $\Mb=m(\Kb)$ equals $m(\s_{B(x,\x')})$ and, by construction, is a non-negative matrix. After this, the sharp G${\aa}$rding inequality produces the eigenvalue estimate, as before.

\subsection{Eigenvalue estimate for  nondegenerate essential spectrum tips.} Estimate \eqref{estimNonhom} can be improved if the boundary point $\pm a$ (or $\pm b$) of the essential spectrum is a nondegenerate extremal point of the function $\kb_0(x)$. Consider, again, the point $a$ and eigenvalues below $a$.
\begin{theorem}\label{ThmNondeg} Suppose that $a=\kb_0(x_0)$ is a nondegenerate minimum of the function $\kb_0$, i.e. for some $x^0\in\pO$,
\begin{equation}\label{min}
    \kb_0(x)\ge a+c|x-x_0|^2,
\end{equation}
for some positive constant $c$ and for all $x\in\pO$ close to $x_0$. Moreover, we suppose that $\kb_0(x)>a$ for all $x\in\pO,$ $x\ne x_0$. Then for the eigenvalues $\l_j^{a,+}$ of $\Kb$ converging  to $a$ from below,
 \begin{equation}\label{Estimate nondegen}
 a-\l_j^{a,+}=O(j^{-\frac1\t}), \t>1.
 \end{equation}

\end{theorem}
\begin{proof} As before, we pass to the operator $\Mb=m(\Kc)$, for which we are going to obtain the estimate for the negative eigenvalues. We localize to  a small neighborhood of $x_0$; outside this neighborhood, the principal symbol of $\Mb$ is positive, separated from zero, and therefore the addition of a pseudodifferential operator of order $-1$ may produce only a finite set of eigenvalues. Near the point $x_0$ the principal symbol of $\Mb$ is a matrix, with the lowest eigenvalue having a nondegenerate zero minimum at $x_0$ and other eigenvalues being positive near $x_0$.

Further on, consider the diagonal matrix $M(x,\x')=M(x)$ (i.e., depending only on the point $x\in\pO$ but not on the covector $\x'$) which equals $\diag(\frac{c}{2}|x-x_0|^2)$. By the condition \eqref{min}, the principal symbol of $\Mb$ is a matrix greater than $M(x)$ for $x$ close to $x_0$, $\s_{\Mb}(x,\x)\ge C M(x)$. Therefore, by the variational principle and the sharp G${\aa}$rding inequality, the counting function for the negative eigenvalues of $\Mb$ is majorated dy the counting function for the negative eigenvalues of the   operator $M(x)-W$, where $W=w(x,D)$ is a pseudodifferential operator of order $-1.$ Further on, as before, we can restrict ourselves to the case of $w(x,D)$ being $c(1-\D)^{-\frac12}$ (the positive constant $c$ here and further on may be different in different formulas and its value is of no importance.)

Generally, spectral estimates for various two-term operators have been studied for a long time, however not for operators of the form $M(x)-W$ which we need. The reasoning to follow reduces this problem to spectral estimates for weighted weakly singular integral operators.

In variational setting, we are interested in the study of the quantity
\begin{gather}\label{var1}
    n(-t,M(x)-W)=\#\{j:\l_j(M(x)-W)<-t\}=\\ \nonumber \max\dim\{\Lc: ((M(x)-W) u,u)<-t\|u\|^2, \, u\in\Lc\setminus\{0\}\},
\end{gather}
as $t\to +0$ (here and further on, the subspaces are considered in $L^2$).
We write the inequality in \eqref{var1} as
 \begin{equation*}
    \int M(x)|u|^2 dx -\int \langle Wu(x),{ u(x)}\rangle dx < -t\int|u(x)|^2 d\s(x),
 \end{equation*}
 or
 \begin{equation}\label{var2}
\int \langle Wu(x),{ u(x)} d\s(x)\rangle > \int (t+M(x))|u(x)|^2 d\s(x).
 \end{equation}
 We use now the classical inequality $yz\le \frac{y^p}{p}+\frac{z^q}{q}$ for arbitrary numbers $y,z>0$, $p^{-1}+q^{-1}=1$. We set here $y=t^{1/p}, z=|x|^{2/q}$, with $p,q\in(1,\infty)$ to be fixed later. Thus, we have
 \begin{equation*}
    t+c|x|^2\ge c t^{1/p}|x|^{2/q}.
 \end{equation*}
 and, therefore,
 \begin{equation}\label{var3}
   \int (t+M(x))|u(x)|^2 d\s(x)\ge C t^{1/p}\int |x|^{2/q}|u(x)|^2 d\s(x).
 \end{equation}
 If we replace the right-hand side in \eqref{var2} by a smaller quantity, namely, by the right-hand side in \eqref{var3}, then the maximal dimension of subspaces, where the resulting inequality holds, can only increase. So,
 taking into account \eqref{var1}, \eqref{var2}, \eqref{var3},
 we obtain an estimate for  the quantity $n(-t,M(x)-W)$ that we interested in:
 \begin{equation}\label{var4}
    n(-t,M(x)-W)\le \max\dim\{\Lc: (Wu,u)>  t^{1/p}\int |x|^{2/q}|u(x)|^2d\s(x), u\ne 0, u\in\Lc\}.
 \end{equation}

 Now, denote $|x|^{\frac1q}u$ by $v$ in \eqref{var4}. We obtain
\begin{equation}\label{var5}
    n(-t,M(x)-W)\le \max\dim\{\Lc: (|x|^{-1/q}(W(|x|^{-\frac1q}v),v)>  ct^{1/p}\int |v(x)|^2 d\s(x), v\ne 0, v\in\Lc\}.
 \end{equation}
 The quantity on the right in \eqref{var5} is, by the variational principle, exactly the number of eigenvalues of the operator $Z=|x|^{-1/q}W|x|^{-1/q}$, which are larger than $ t^{1/p}$. This operator, up to weaker terms which do not contribute to the eigenvalue estimates, has the form

\begin{equation*}%\label{intOper}
   (Zu)(x)=\int |x|^{-\frac12}|x-y|^{-1}|y|^{-\frac12}v(y)d\s(x).
\end{equation*}

 We need now to recall the estimate of singular numbers of integral operators with weighted difference kernels, obtained in \cite{BSInt}. We cite here Theorem 10.3 a), from \cite{BSInt}, for our special case.

\textbf{ Theorem 10.3.} Let $X$ be a measurable set in $\R^m$,  $m=2$, $F=|x-y|^{k}$ $k=-1$; let $\bb(x)$, $\cb(x)$ be measurable functions belonging to $L^{r}(X)$, $2r^{-1}=\d^{-1}$, $\d^{-1}=1+km^{-1}=\frac12$. Then for the integral operator $Z$ in $L^2(X)$ with kernel $\bb(y)F(x-y)\cb(x)$ the estimate holds

\begin{equation}\label{BSEst}
    s_n(Z)\le C n^{-\frac1\d}\|\bb\|_{L^r}\|\cb\|_{L^r},
\end{equation}
with constant $C$ depending only on $m,\d,r$.

We apply this theorem for $F(x-y)=|x-y|^{-1}$, $\bb,\cb=|x|^{-\frac1q}$.
To satisfy the conditions of the theorem, we need $\bb,\cb\in L^r$, $r=4$. This can be achieved by requiring $q>2$.

Thus, \eqref{BSEst} takes the form
\begin{equation}\label{BsSpec}
    s_n(Z) \le C(q) n^{-\frac12}.
\end{equation}
The estimate \eqref{BsSpec} can be written in terms of the counting function for singular numbers,
\begin{equation*}
    n(s, Z)\le C'(q)s^{-2}.
\end{equation*}
We set here $s=t^{\frac1p}$, which gives, by \eqref{var4}
\begin{equation}\label{EstFIN}
    n(-t, M(x)-W)\le n(t^{\frac1p}, Z)\le C' t^{-\frac2p}.
\end{equation}
Since $q>2$ could be taken arbitrarily close to $2$, $p<2$ can be also taken arbitrarily close to $2$. Therefore, the exponent $\t=\frac2p$ can be an arbitrary number larger than 1, which is equivalent to the required estimate \eqref{Estimate nondegen}.
\end{proof}
\begin{remark} The estimate \eqref{Estimate nondegen}  can be somewhat improved to $a-\l_j^{a,+}=O(j^{-1}\log j)$ by tracing the dependence of the constant $C(q)$ in \eqref{BsSpec} on the exponent $q$ as $q$ approaches 2.
\end{remark}

So, we see that in the case of a nondegenerate minimum of the function $\kb_0(x)$, the eigenvalues of the NP operator converge almost twice as fast compared with the case of  constant Lam\'e parameters.

We suspect that a sharper estimate of the form \eqref{Estimate nondegen} with $\t=1$ does not hold. The reason for this is the circumstance that for $q=2$ the operator $Z$ in the proof of  Theorem \ref{ThmNondeg} is not even bounded (this is exactly the case when the Hardy type inequality fails, see the reasoning in \cite{Yafaev}).

If the minimum of $\kb_0(x)$ at the point $x_0$ at the boundary is \emph{not nondegenerate}, then the rate of convergence of the eigenvalues to their limit point is determined by the order of zero of $\kb_0(x)$ at $x_0.$ If, in particular, $\kb_0(x)-a\ge C |x-x_0|^{2\n}$, $\n$ being a positive integer, $\n>1$, the reasoning similar to the one in Theorem \ref{ThmNondeg} shows that
\begin{equation}\label{EstFINl}
    \l_j^{a,+}-a=O(j^{-\frac1\t}), \t>\frac{2\n-1}{\n}.
\end{equation}
So, as $\n$ grows, i.e. the function $\kb_0$ becomes more and more flat at $x_0$, the estimate \eqref{EstFINl} approaches the general estimate \eqref{estimNonhom}.

For the modified NP operator, since the eigenvalues of the principal symbol are constants, the convergence rate of eigenvalues is the same as in Theorem \ref{rate of convergence}.

\section{The case $d=2$ and further discussion}
The above considerations extend almost automatically to the case of the elastic problem in two dimensions. The results are similar, with natural modifications. The essential spectrum of the NP operator coincides with the range of the functions $\pm\kb_0$ on the boundary, and thus consists of an even number of symmetrical with respect to the zero point closed intervals, some of which can degenerate to single points. The modified NP operator, as in \eqref{modif}, is polynomially compact with polynomial $p(s)=s^2-1$ and its essential spectrum consists of two points $\pm 1$. The case of the two-dimensional Lam\'e system with homogeneous material  was considered in \cite{2D}. Recently,  the rate of convergence of eigenvalues was studied, including the case of finite smoothness of the boundary, see \cite{AKM}. It turns out that the situation with eigenvalue behavior is rather similar to the one for the electrostatic NP operator \cite{AKM2}.
In fact, there exist  two sequences  of eigenvalues $\lambda_j^{\pm}(\Kb)$ accumulating to $\pm \kb_0$ and their behavior is strongly related the smoothness of the boundary:
\begin{theorem}[Polynomial decay rates for elastic NP eigenvalues, \cite{AKM}]\label{polynomial decay}
Let $\Omega\subset {\mathbb R}^2$ be a $C^{k, \alpha}$ {\rm($k\geqq 2$)} region. For any $d>-(k+\alpha)+3/2$,
\begin{equation}
\lambda_j^{\pm}(\Kb)\mp \kb_0  = o(j^{d}) \quad \text{as}\; j\rightarrow \infty .
\end{equation}
\end{theorem}Furthermore, for the analytic curves,
\begin{theorem}[Exponential decay rates for elastic NP eigenvalues, \cite{AKM}]\label{analytic decay}
Suppose that $\G$ is real analytic. Let $q$ be a parametrization of $\G$ by a Riemann mapping and let $\epsilon_q$ be its modified maximal Grauert radius. For any $\epsilon <\epsilon_q/8$,
\begin{equation}\label{expdecay}
\lambda^{\pm}_j(\Kb) \mp \kb_0 = o(e^{-\epsilon j}) \quad \text{as}\; j\rightarrow \infty
\end{equation}
 Here we simply refer to \cite{AKM} for the precise meaning of the modified maximal Grauert radius $\epsilon_q$.
\end{theorem}
Thus one can evaluate  the eigenvalue  decay rate for smooth curves in two dimensions. In the special example of an ellipse, which is a  $C^{\infty}$ smooth curve the general estimate gives $\lambda_j^{\pm}(\Kb)\mp \kb_0=O(j^{-\infty})$.
We emphasize that  ellipses are analytic curves and so the eigenvalues have stronger decay properties than in the case of finitely smooth and even of general infinitely smooth curves, that is, $\lambda_j^{\pm}(\Kb)\mp \kb_0=O(e^{-\epsilon j})$. It is worth mentioning here that
the explicit eigenvalue asymptotics of the ellipse can be found, see \cite{2D}:
\begin{align*}
|\lambda_j^{+}(\Kb)- \kb_0| &\sim \frac{1}{(\lambda+2\mu)\tau}j\left(\frac{a-b}{a+b} \right)^{-j}, \\
|\lambda_j^{-}(\Kb)+ \kb_0| &\sim \frac{(\lambda+\mu)(\lambda+3\mu)}{4\mu^2(\lambda+2\mu)\tau}j\left(\frac{a-b}{a+b} \right)^{-2j},
\end{align*}
where the parameters $a, b$ denote the semi-major and semi-minor axes respectively and $\tau$ is the eccentricity of the ellipse.
This example shows that the convergence rate (\ref{expdecay}) is not optimal.
We also notice that the decay rates for the eigenvalues converging to the points $ \kb_0$ and $-\kb_0$ are different. These mysterious phenomena for general analytic curves are still unresolved.

As for the three dimensional case, we know explicit eigenvalues only for the sphere \cite{DLL}. The three sequences of eigenvalues are given by
\begin{align*}
\lambda_j^{0}(\Kb)&=\frac{3}{2(2j+1)}, \\
\lambda_j^{+}(\Kb)&=\frac{3\lambda-2\mu(2j^2-2j-3)}{2(\lambda+2\mu)(4j^2-1)}, \\
\lambda_j^{-}(\Kb)&=\frac{-3\lambda+2\mu(2j^2+2j-3)}{2(\lambda+2\mu)(4j^2-1)}.
\end{align*}
Note that the eigenvalues $\lambda_j^{0}(\Kb)$ in the first series, accumulating to $0$, are independent of the Lam\'e parameters. Moreover their multiplicities are $2j+1$ and so the partial sum is $3/2$. This fact is, probably,  related to the so called  $1/2$ conjecture for {\it electrostatic} NP operators \cite{AKMU}.  In other words, for each positive integer $j$ there are $2j+1$ {\it electrostatic} NP eigenvalues whose sum is $1/2$. Martensen \cite[Theorem 1]{Ma} proved that this holds to be true for the {\it electrostatic} NP eigenvalues on ellipsoids. Regarding the elastic NP operators, the point  $0$ in the essential spectrum is an universal constant and one may expect that the corresponding eigenvalues (or, at least, their asymptotic behavior) are independent of Lam\'e parameters. The partial sums is, probably, equal to $3/2$ at least on ellipsoids. However there are no proofs for the $3/2$ conjecture.

One more interesting problem is the relation between  eigenvalues and divergence free fields:
$$
\div u=0.
$$
If $\mu$ is constant, then the Lam\'e equation \eqref{Lame} becomes
$$
\triangle u=0.
$$
Thus the Lam\'e parameters disappear from the equation and the eigenfunctions for the eigenvalues accumulating to $0$ are expected to correspond to the subspace of the divergence free fields. But we don't know how to characterize the linear hull of each sequences.

The results in \cite{AKM} are essentially based upon the possibility, in the 2D case, of splitting the NP operator, up to an arbitrary negative order, into the direct sum of  scalar operators with constant principal symbol, i.e., of shifted electrostatic NP operators. Such reduction is possible thanks to a conveniently very simple topology of the cospheric bundle of the one-dimensional boundary.
It is rather tempting to perform a similar splitting in the 3D case as well, thus reducing the 3D elastic NP operator to a direct sum of three electrostatic operators. However, at the moment, topological obstacles seem to prevent one from doing this, unless the Euler characteristics of the surface is zero. We expect, however, that for the case of a homogeneous body, results on the asymptotics of the eigenvalues can be obtained, similar to the ones in \cite{M}, \cite{MR}, where for the electrostatic NP operator important relations of such asymptotics with geometrical characteristics  of the surface have been found.

Finally, we hope to be able to reduce the regularity requirements in the results on the position of the essential spectrum and on the $j^{-\frac12}$ eigenvalues convergence rate estimate to boundaries of the class $C^{2,\a}$ or even to the ones of class $C^{1,1}$. In the 2D case this was obtained in \cite{AKM} by a representation of the NP operator as a pseudodifferential operator plus a Hilbert Schmidt-one. In our case we suppose that an approximation by smooth surfaces, similar to the one used in \cite{RT1}, \cite{RT2} should take care of the related complications. We will pursue these issues in the future.

\end{document}